\documentclass{amsart}
\usepackage[utf8]{inputenc}
\usepackage{amsmath}
\usepackage{amssymb}
\usepackage{graphicx}
\usepackage{amsfonts}

\theoremstyle{plain}
\newtheorem{theorem}{Theorem}[section]
\newtheorem{lemma}[theorem]{Lemma}
\newtheorem{claim}[theorem]{Claim}

\newtheorem{proposition}[theorem]{Proposition}

\newtheorem{question}[theorem]{Question}

\theoremstyle{definition}\newtheorem{definition}[theorem]{Definition}

\theoremstyle{definition}\newtheorem{example}[theorem]{Example}
\theoremstyle{definition}\newtheorem{remark}[theorem]{Remark}

\numberwithin{equation}{section}

\newcommand{\R}{\mathbb{R}}

\newcommand{\om}{\omega}

\newcommand{\bd}{\begin{definition}}
\newcommand{\ed}{\end{definition}}

\DeclareMathOperator{\graph}{graph}
\DeclareMathOperator{\dom}{dom}
\DeclareMathOperator{\supp}{supp}
\newcommand{\mc}{\mathcal}

\begin{document}

\title{Naively Haar null sets in Polish groups}

\author[M\'arton Elekes]{M\'arton Elekes$^\ast$}
\thanks{$^\ast$Partially supported by the
Hungarian Scientific Foundation grants no.~104178 and no.~113047.}

\author[Zolt\'an Vidny\'anszky]{Zolt\'an Vidny\'anszky$^\ddag$}
\thanks{$^\ddag$Partially supported by the
Hungarian Scientific Foundation grants no.~104178 and no.~113047.}

\insert\footins{\footnotesize{MSC codes: Primary 03E15, 54H05; Secondary 54H11, 28A99, 03E17, 22F99}}
\insert\footins{\footnotesize{Key Words: Polish groups, non-locally compact Polish group, Haar null, Christensen, shy, prevalent, universally measurable, Fremlin, Problem FC}}

\begin{abstract}
Let $(G,\cdot)$ be a Polish group. We say that a set $X \subset G$ is \emph{Haar null} if there exists a universally measurable set $U \supset X$ and a Borel probability measure $\mu$ such that for every $g, h \in G$ we have
$\mu(gUh)=0$. We call a set $X$ \emph{naively Haar null} if there exists a Borel probability measure $\mu$ such that for every $g, h \in G$ we have $\mu(gXh)=0$. 

Generalizing a result of Elekes and Stepr\=ans, which answers the first part of Problem FC from Fremlin's list, we prove that in every abelian Polish group there exists a naively Haar null set that is not Haar null.

\end{abstract}
\maketitle

\section{Introduction}
Let $(G,\cdot)$ be a Polish group. It is well known that there exists a left Haar measure on $G$ (that is, a
regular left invariant Borel measure that is finite for compact sets and positive
for non-empty open sets) if and only if $G$ is locally compact. It can be proved that the ideal of left Haar measure zero sets does not depend on the choice of the measure, moreover, it coincides with the ideal of the right Haar null sets (that can be defined analogously). This ideal plays an important role in the study of locally compact groups and there are a lot of interesting non-locally  compact groups, so it is very natural to try to construct well-behaved generalizations of this notion in non-locally compact groups.

Christensen \cite{christ} suggested a generalization, which is widely used in diverse areas of mathematics. We will call a set
\emph{universally measurable} if it is measurable with respect to every Borel
probability measure and we identify Borel measures with their completions.

\begin{definition}
 A set $X \subset G$ is called \emph{Haar null} if there exists a universally measurable
  set $U
\supset X$ and a Borel probability measure $\mu$ on $G$ such that $\mu(gUh)=0$
for every $g, h \in G$.
\end{definition}

For a Haar null set we will call a corresponding measure a \emph{witness measure} and the set $U$ its \emph{universally measurable hull}. Notice that some authors use a slightly more restrictive notion, namely they require the hull to be a Borel set. These two notions provably differ in non-locally compact abelian Polish groups (see \cite{mienk}, where the above notion is called generalized Haar null), although in practice this makes little difference as the studied sets are typically Borel.

Christensen proved that in a locally compact Polish group a set is Haar null if and only if it is of measure zero with respect to a (or equivalently, every) Haar measure. He also showed that the collection of Haar null sets form a $\sigma$-ideal in every Polish group. 

Our paper is motivated by the first part of Problem FC on Fremlin's list \cite{frem}. The problem is whether we really need the universally measurable hulls in this definition. Let us consider the following notion.

\begin{definition}
 A set $X \subset G$ is called \emph{naively Haar null} if there exists a Borel probability measure $\mu$ on $G$ such that $\mu(gXh)=0$
for every $g, h \in G$.
\end{definition}

Using this terminology Fremlin's problem asks whether every naively Haar null set is Haar null. This question was answered by Elekes and Stepr\=ans \cite{elekstep}. Notice that this question makes sense in any uncountable Polish group, though the original question was formulated in $\R$.

It was observed by Dougherty \cite{dough} that under the Continuum Hypothesis (CH) the answer is negative in the groups of the form $G \times G$. In fact, it is easy to see that if we consider a well-ordering $<_W$ of $G$ in order type $\omega_1$ as a subset $W$ of $G \times G$ then both
$W$ and $(G \times G) \setminus W$ are naively Haar null. In particular, since $G \times G$ is clearly not Haar null and Haar null sets form a $\sigma$-ideal, we obtain that either $W$ or its complement is a naively Haar null, non-Haar null set. 

Elekes and Stepr\=ans proved that in $\R^n$ the assumption of CH can be dropped. In this paper we extend their result to every abelian Polish group, proving the following statement.

\begin{theorem}
\label{t:main}
 Let $G$ be an uncountable abelian Polish group. There exists a subset of $G$ that is naively Haar null but not Haar null. 
\end{theorem}

We have to treat the case of locally compact and non-locally compact Polish topological groups separately. We start with the locally compact case, which is essentially a transfinite construction, while to solve the non-locally compact case we use ideas from \cite{mienk}.

In fact, we will prove slightly more in both cases. A natural modification of the definition of Haar nullness that was investigated by several authors (\cite{rosendal}, \cite{solecki} etc.) is the following:

\begin{definition}
 A set $X \subset G$ is called \emph{left Haar null} if there exists a universally measurable
  set $U
\supset X$ and a Borel probability measure $\mu$ on $G$ such that $\mu(gB)=0$
for every $g \in G$.
\end{definition}

The naive version of this notion can be defined analogously. It is easy to see using convolution that in locally compact groups a set is left Haar null if and only if it is of measure zero with respect to a Haar measure.

In the locally compact case we show that every Polish group has a naively left Haar null set that is not Haar null.

In the non-locally compact case our results  (including of course the part cited from \cite{mienk}) can be generalized to every non-locally compact Polish group that admits a two-sided invariant metric.

\section{Preliminaries} 

We use the notation of \cite{cdst} and \cite{mienk}. 

A \emph{Polish group} is a topological group whose topology is Polish (a topology is called Polish if it is separable and completely metrizable). If $G$ is a Polish group we denote by $\mc{K}(G)$ the Polish space of non-empty compact subsets of $G$ (endowed with topology given by the Hausdorff metric). 
$\mathcal{P}(G)$ stands for the set of Borel probability measures on $G$ (i.~e.~the completions of probability measures
defined on the Borel sets). With the weak*-topology these measures form a Polish space. 

For $\mu \in
\mathcal{P}(G)$ we denote by $\supp(\mu)$ the support of $\mu$, that is, the smallest closed set $F$ such that $\mu(G \setminus F)=0$. The collection of Borel probability measures on $G$ with compact support is denoted by $\mathcal{P}_c(G)$. It can be shown that $\mc{P}_c(G)$ is a Borel subset of $\mc{P}(G)$ \cite[Lemma 3.2]{mienk}.

A simple but important observation is that if we have a Haar null set with witness measure $\mu$ then passing to a $\mu$-positive compact subset (\cite[Theorem 17.10]{cdst}) and normalizing we can obtain a witness measure with compact support, thus:
\vspace{-0.3mm}
\begin{equation}
\text{Every Haar null set has a witness measure with compact support.}
\label{e:cpt}
\end{equation}
\vspace{-0.1mm}
If $G$ is locally compact we will denote a left Haar measure on $G$ by  $\lambda$. 

As usual, $\mathbf{\Sigma}^0_\xi(G)$ and $\mathbf{\Pi}^0_\xi(G)$ stand for the $\xi$th additive and multiplicative level of
the Borel hierarchy in $G$. If it is clear from the context, we will omit the underlying Polish space. Continuous images of Borel sets are called analytic sets. A $\mathbf{\Sigma}^0_\xi(G)$ set $C$ is called \emph{$\mathbf{\Sigma}^0_\xi$-complete} if for every Polish space $X$ and $A \in \mathbf{\Sigma}^0_\xi(X)$ there exists a continuous map $f: X \to G$ such that $f^{-1}(C)=A$. It is not hard to prove that in an uncountable Polish space $X$ a $\mathbf{\Sigma}^0_\xi(X)$-complete set must be in 
$\mathbf{\Sigma}^0_\xi(X) \setminus \mathbf{\Pi}^0_\xi(X)$.

If $H \subset X \times Y$ and $x \in X$ then the $x$-section of $H$ is the set $H_x = \{y \in Y :
(x,y) \in H \}$. For a function $f \colon X \times Y \to Z$
the $x$-section is the function $f_x \colon Y \to Z$ defined by $f_x(y) =
f(x,y)$. 

For $A,B \subset G$ let $A \cdot B$ denote the Minkowski product of $A$ and $B$, i.~e.~the set $\{a\cdot b:a \in A, b\in B\}$, while $A^{-1}$ stands for the set $\{a^{-1}:a \in A\}$.

If $H < G$ then a \textit{partial left  (right) transversal to $H$} is a subset of $G$ that intersects every left (right) coset of $H$ in at most one point.

\section{Locally compact groups}
\begin{theorem}
\label{t:lc}
 
Every uncountable locally compact Polish group contains a naively left Haar null set that is not Haar null. In particular, every uncountable abelian Polish group contains a naively Haar null set that is not Haar null.
\end{theorem}

 First we prove an easy lemma.
  \begin{lemma}
  \label{l:easy} Let $P \subset G$ be a perfect set and $B$ be a Borel $\lambda$-positive set. Then there exists $g \in G$ such that $|gP \cap B|=\mathfrak{c}$.
 \end{lemma}
 \begin{proof}
 Suppose the contrary. The indirect assumption and the Borelness of the set $B \cap gP$ imply that for every $g \in G$ we have $|B \cap gP| \leq \aleph_0$. Fix a continuous probability measure $\nu$ with $\supp(\nu) \subset P$. Then for every $g \in G$ we have \[|gB \cap \supp(\nu)| \leq |gB \cap P|=|B \cap g^{-1}P| \leq \aleph_0,\]
 
so $\nu(gB)=0$. Thus $B$ is a Borel left Haar null set, consequently a $\lambda$-null set, a contradiction.
 \end{proof}
 
 The following statement is the crucial point of our argument.
 
\begin{lemma}
\label{l:basic}
Suppose that there exists an uncountable analytic subgroup $H<G$ such that $|G:H|>\aleph_0$.
\begin{enumerate}
 \item There exists a set $X \subset G$ that is naively left Haar null, but not $\lambda$-null.
 \item If $H$ is also a normal subgroup then there exists a set $X$ that is naively Haar null, but not $\lambda$-null.
\end{enumerate}

\end{lemma}
\begin{proof}
 First we prove that if $X$ is a partial left transversal to $H$ then $X$ is naively left Haar null and if additionally $H$ is a normal subgroup then a partial transversal is Haar null. 
 
 Since $H$ is an uncountable analytic set it contains a perfect subset, so we can fix a continuous Borel probability measure $\mu$ with $\supp(\mu) \subset H$. Now, for an arbitrary $g \in G$ we have \[|gX \cap \supp(\mu)|= |X \cap g^{-1}\supp(\mu)| \leq  |X \cap g^{-1}H| \leq 1,\]
 so $\mu(gX)=0$.
 
 If $H$ is a normal subgroup then being a left or a right transversal is the same and repeating this argument we obtain that for every $g,h \in G$
 \[|gXh \cap \supp(\mu)|= |X \cap g^{-1}\supp(\mu)h^{-1}| \leq  |X \cap g^{-1}Hh^{-1}| =\]
 \[=|X \cap g^{-1}h^{-1}hHh^{-1}|=|X \cap g^{-1}h^{-1}H| \leq 1,\]
 which shows that $X$ is naively Haar null.
 
 So it is enough to construct a non-$\lambda$-null partial left transversal to $H$. Notice first that since $H$ has the Baire property, the condition $|G:H|>\aleph_0$ implies that $H$ is meager, hence by \cite[p. 5]{kechbeck} there exists a partial perfect left transversal $P$ to $H$. 
 
 Enumerate the $\lambda$-null $G_\delta$ sets as $\{G_\alpha:\alpha<\mathfrak{c}\}$. Suppose that we have already constructed a partial left transversal to $H$, denoted by $\{x_\beta:\beta<\alpha\}$, such that $x_\beta \not \in G_\beta$ for every $\beta<\alpha$.
 
 \textit{Claim.} There exists an $x_\alpha$ such that 
 \[x_\alpha \not \in \bigcup_{\beta<\alpha} x_\beta H \cup G_\alpha\]
 and this is clearly enough to continue the induction.
  
 Suppose the contrary. Then the set $S=\bigcup_{\beta<\alpha} x_\beta H$ is a co-null set. 
 
 Now notice that for every $g \in G$ we have $|gP \cap S|<\mathfrak{c}$: otherwise there would exist an ordinal $\beta<\alpha$ such that \[gp_1=x_\beta h_1 \text{ and } gp_2=x_\beta h_2\]
 for some distinct $p_1,p_2 \in P$ and $h_1,h_2 \in H$. But then \[p^{-1}_2 p_1=h_2^{-1}h_1 \in H,\]
 which contradicts the fact that $P$ was a partial left transversal. 
 
 So every left translate of $P$ intersects $S$ in less then $\mathfrak{c}$ many points and since $S$ is co-null, it contains a co-null $F_\sigma$ set $B$. By Lemma \ref{l:easy} this is impossible, finishing the proof of the claim. 
 
We have the claim, thus the induction can be carried out. The resulting set $X=\{x_\alpha:\alpha<\mathfrak{c}\}$ is a partial left transversal to $H$ which contains a point outside of every $\lambda$-null $G_\delta$ set, consequently it is a non-$\lambda$-null set, so we are done. 
 \end{proof}

 \begin{lemma}
 \label{l:farah}
  Every uncountable Polish group $G$ contains an uncountable Borel subgroup of uncountable index.
 \end{lemma}
\begin{proof}
 By \cite{farah} every uncountable Polish group contains a $\mathbf{\Sigma}^0_3$-complete subgroup $H < G$. $|G:H| \leq \aleph_0$ would imply that 
 $H=G \setminus (\bigcup_{n \in \omega} g_n H)$ for some $g_n \in G$. But the set $\bigcup_{n \in \omega} g_n H$ is $\mathbf{\Sigma}^0_3$ so $H$ would be a $\mathbf{\Pi}^0_3$ set which contradicts the $\mathbf{\Sigma}^0_3$-completeness of $H$.
 
 In order to see that $H$ is uncountable just notice that every countable set is $\mathbf{\Sigma}^0_2$, so a countable set cannot be $\mathbf{\Sigma}^0_3$-complete.
\end{proof}
\begin{remark}
 It is not hard to prove Lemma \ref{l:farah} directly, constructing by induction a perfect scheme so that the corresponding perfect set is compact and very ``thin'': it generates a $\sigma$-compact first category subgroup of $G$. Compactness implies that the generated subgroup is also $\sigma$-compact. 
\end{remark}

Putting together Lemmas \ref{l:farah} and \ref{l:basic} we obtain Theorem \ref{t:lc}, thus finishing the proof of Theorem \ref{t:main} in the locally compact case. 

\section{Non-locally compact groups}

In this section we will use ideas and a large part of the proof from \cite{mienk}. Unfortunately, this cannot be avoided, since, apart from Proposition \ref{p:eltol_haar}, the used proof segment is not explicitly citable. 

\begin{theorem}

\label{t:nlc}
 
 Suppose that $G$ is a non-locally compact abelian Polish group. Then there exists a set $X$ that is naively Haar null but not Haar null.
\end{theorem}
We will use the following proposition from \cite{mienk}.
\begin{proposition}
 \label{p:eltol_haar}
 Let $C \in \mathcal{K}(G)$ be fixed. Then there exists a Borel map $t:
\mathcal{K}(G) \times 2^\om \times 2^\om \to G$ so that
\begin{enumerate}
\item\label{p:eltol1_haar} if $(K,x,y) \not = (K',x',y')$ are elements of $\mathcal{K}(G) \times 2^\om \times 2^\om$ then
\[
(K \cdot C^{-1} \cdot t(K,x,y)) \cap (K'  \cdot C^{-1} \cdot t(K',x',y'))= \emptyset
\]
\item\label{p:eltol2_haar} for every $K \in \mathcal{K}(G)$ and $y \in 2^\om$  the
  map $t(K,\cdot, y)$ is continuous.
\end{enumerate}
\end{proposition}
Note that the original proposition in \cite{mienk} contains a typo, namely, $C'$ should be read as $C$.
The following Proposition is the analogue of \cite[Theorem 3.1.]{mienk}, but the idea of the proof is different.
\begin{proposition}
\label{p:Bor_haar}
Let us denote the usual measure on $2^\omega$ by $\lambda$. There exists a partial function
$f:\mathcal{P}_{c}(G) \times 2^\om \to G$ satisfying the
following properties: $\forall \mu \in \mathcal{P}_c(G)$
\begin{enumerate}
 \item\label{p:Bor1_haar} $(\forall x \in 2^\om)\left[ (\mu, x) \in \dom(f) \Rightarrow f(\mu,x) \in \supp (\mu)\right]$,
 \item\label{p:Bor2_haar} $(\forall S \in \mathbf{\Pi}^0_2(2^\om \times G))\left[(\graph(f_\mu)
     \subset S) \Rightarrow ((\lambda \times \mu)(S)>0 )\right].$
\end{enumerate} 
\end{proposition}

\begin{proof}
 Let $\mu \in \mc{P}_c(G)$ be fixed. We claim that there exists a partial function $f_\mu:2^\om \to \supp(\mu)$ such that whenever $S \in \mathbf{\Pi}^0_2(2^\om \times G)$ contains  $\graph(f_\mu)$ then $(\lambda \times \mu)(S)>0.$ This is enough to prove the proposition:
 the function defined as $f(\mu,x)=f_\mu(x)$ clearly satisfies the requirements.
 
 We construct the function $f_\mu$ by transfinite induction. Let $\{G_\alpha:\alpha<\mathfrak{c}\}$ be an enumeration of the $\lambda \times \mu$-null $G_\delta$ subsets of $2^\om \times G$. At stage $\alpha$ we choose a pair $(x_\alpha,f_\mu(x_\alpha)) \in (2^\om \times \supp(\mu)) \setminus G_\alpha$ such that $\{(x_\beta,f_\mu(x_\beta)):\beta \leq \alpha\}$ is a graph of a partial function. This can be done, since if $G_\alpha$ is of $\lambda \times \mu$ measure zero
 then $\lambda$ almost all vertical sections of $G_\alpha$ have $\mu$ measure zero. Therefore, the set $H=\{x_\beta:\beta<\alpha\} \cup \{x: \mu((G_\alpha)_x)>0\}$ is
 not the whole $2^\om$: otherwise, as the complement of the set $\{x_\beta:\beta<\alpha\}$ would be a $\lambda$-null set, $\{x_\beta:\beta<\alpha\}$ would be a $\lambda$-positive measurable set, which must contain a non-empty perfect set, contradicting that $|\alpha|<\mathfrak{c}$. Consequently, we can choose $x_\alpha \not \in H$ and $f_\mu(x_\alpha)  \in \supp(\mu) \setminus (G_\alpha)_{x_\alpha}$, which shows that the induction can be carried out.
\end{proof}
In order to prove Theorem \ref{t:nlc} we need the following simple observation, which is probably well-known.

\begin{lemma}
\label{l:univ}
Suppose that $X$ and $Y$ are Polish spaces, $U \subset Y$ is a universally measurable set and $f:X \to Y$ is a continuous function. Then $f^{-1}(U)$ is also universally measurable.
\end{lemma}

\begin{proof}
Let $\mu$ be a Borel probability measure on $X$. Let us denote by $\nu$ the push-forward measure on $Y$, i.~e.~the measure defined by $\nu(B)=\mu(f^{-1}(B))$ for  every Borel set $B \subset Y$. It is easy to see that $\nu$ is a Borel probability measure on $Y$.
Consequently, as $U$ is $\nu$ measurable we have $U=B \cup N$ with some Borel set $B$ and $\nu(N)=0$. Then $f^{-1}(U)=f^{-1}(B) \cup f^{-1}(N)$ and $f^{-1}(B)$ is Borel while $\mu(f^{-1}(N))=\nu(N)=0$.

Thus $f^{-1}(U)$ is $\mu$-measurable, which finishes the proof. \end{proof}

We continue with the proof of Theorem \ref{t:nlc}, which is a word-by-word repetition of the proof of \cite[Theorem 4.1]{mienk} mutatis mutandis.
\begin{proof}[Proof of Theorem \ref{t:nlc}]
  Let $f$ be given by Proposition \ref{p:Bor_haar}.

It can be proved that the map $\mu \mapsto \supp(\mu)$ from $\mathcal{P}_c(G)$ to $\mathcal{K}(G)$ is Borel, see \cite[17.38]{cdst}. Let us denote this map by $\supp$. 
Let us also fix a Borel bijection $c:\mathcal{P}_c(G) \to 2^\omega$ (which we
think of as a coding map) and a continuous Borel probability measure
$\nu$ on $G$ with compact support $C$. Let $t:\mathcal{K}(G) \times 2^\om \times
2^\om \to G$ be the map from Proposition \ref{p:eltol_haar} with the $C$ fixed above, and define the map $\Psi \colon \mathcal{P}_c(G) \times 2^\om \times
G \to G$ by
\begin{equation}
\label{e:psi_haar}
\Psi(\mu,x,g)=g \cdot t(\supp(\mu),x,c(\mu)).
\end{equation}
Finally, define $X=\Psi(\graph(f)) \subset G$.

\begin{claim}
\label{cl:c_haar}
 $X$ is naively Haar null.
\end{claim}

\begin{proof}
We prove that $\nu$ is witnessing this fact. Actually, we prove more: $|C \cap Xg|\leq 1$ for every $g \in G$, or equivalently
$|Cg \cap X|\leq 1$ for every $g \in G$. So let us fix $g \in G$.
\[
X = \Psi(\graph(f)) = \{ \Psi(\mu, x, f(\mu, x)) : (\mu, x) \in \dom(f) \} = 
\]
\[
\{f(\mu, x) \cdot t(\supp(\mu), x, c(\mu)) : (\mu, x) \in \dom(f) \},
\]
hence the elements of $X$ are of the form $g^{\mu,x} = f(\mu, x) \cdot
t(\supp(\mu), x, c(\mu))$. This element $g^{\mu,x}$ is clearly in $A^{\mu,x} =
\supp(\mu) \cdot t(\supp(\mu), x, c(\mu))$ by Property (\ref{p:Bor1_haar}) of
Proposition \ref{p:Bor_haar}, and the sets $A^{\mu,x}$ form a pairwise disjoint
family as $(\mu, x)$ ranges over $\dom(f)$, by Property (\ref{p:eltol2_haar}) of
Proposition \ref{p:eltol_haar}. Hence it suffices to show that $Cg$ can intersect at
most one $A^{\mu,x}$. But it can actually intersect at most one set of the
form $K \cdot t(K, x, y)(=\supp(\mu) \cdot t(\supp(\mu), x, y))$, since otherwise $g$ would be in the intersection of two
distinct sets of the form $K \cdot C^{-1} \cdot t(K, x, y)$, contradicting Property
(\ref{p:eltol2_haar}) of Proposition \ref{p:eltol_haar}.
\end{proof}

\begin{claim}
 \label{cl:c3_haar}
 $X$ is not Haar null.
\end{claim}

Suppose the contrary, then by definition $X$ has a universally measurable Haar null hull $U$. Then using observation (\ref{e:cpt}) there exists a Borel probability measure $\mu$ with compact support witnessing that $U$ is Haar null. The section map $\Psi_{\mu} = \Psi(\mu,\cdot,\cdot)$ is continuous by \eqref{e:psi_haar} and Property (\ref{p:eltol2_haar}) of Proposition \ref{p:eltol_haar}. Now let $T=\Psi^{-1}_{\mu}(U)$, then $T$ is universally measurable by Lemma \ref{l:univ}. Notice that by the definition of $X$, $f$ and $T$ we have that $\graph(f_\mu) \subset T$. 

We show that for every $x \in 2^\om$ we have $\mu(T_{x})=0$. Suppose the contrary. By the
definition of $T$ we have that \[\Psi(\mu,x,T_x)=\Psi_\mu(\{x\} \times T_x) \subset \Psi_\mu(T)=
U.\] But $\Psi(\mu,x,\cdot):G \to G$ is a translation, so a translate of $U$
contains $T_x$, which is of positive $\mu$-measure, contradicting that $U$ is
Haar null with witness $\mu$.

Hence every vertical section of $T$ has $\mu$-measure zero and since $T$ is universally measurable, by the Fubini theorem $\lambda \times \mu (T) =0$. Therefore, by the regularity of Borel measures, there exists a $\mathbf{\Pi}^0_2(2^\omega \times G)$ set $S \supset T$ of $\lambda \times \mu$-measure zero.
But 

\[\graph(f_\mu) \subset T \subset S\]

which contradicts Property (\ref{p:Bor2_haar}) of
Proposition \ref{p:Bor_haar}. 

\end{proof}

This concludes the proof of Theorem \ref{t:nlc} and hence of Theorem \ref{t:main}.
\section{Remarks and open problems}

We finish our paper with collecting the most important open problems. Fremlin's problem in full generality remains unsolved.

\begin{question}
 Does there exist a naively Haar null non-Haar null set in every uncountable Polish group?
\end{question}

Since we have shown that the existence of an uncountable Borel normal subgroup of uncountable index implies the existence of a naively Haar null non-Haar null set, it would be interesting to know, which groups contain such a subgroup.

\begin{question}
 Does every uncountable non-archimedean compact group contain an uncountable Borel (or analytic) normal subgroup of uncountable index?
\end{question}

In order to show that the notion of naively Haar null sets is indeed very naive, it would be sufficient to prove that naively Haar null sets do not form a $\sigma$-ideal. 

\begin{question}
 Is it true that in every uncountable Polish group $G$ the naively Haar null sets do not form a $\sigma$-ideal? In particular, does there exist a naively Haar null set $S \subset G$ such that
 $G \setminus S$ is also naively Haar null?
\end{question}

As it was mentioned before, under CH the groups of the form $G \times G$ can be decomposed into two naively Haar null sets. A natural attempt to use the same idea in locally compact groups to obtain a ZFC result would be the following.

Let $S \subset G$ be a $\lambda$-positive set of cardinality $non(\mc{N})$ (i.~e.~the smallest cardinality of a positive set). Then a well-ordering of $S$ regarded as a subset $W$ of $S \times S$ is a naively Haar null set with witness $\lambda$, as every vertical section is $\lambda$-null, similarly $(S \times S) \setminus W$ is also naively Haar null. It is easy to see that $S \times S$ is $\lambda \times \lambda$-positive, however it is not clear, whether this set is naively Haar null. So the following question arises: does there exist a set $S \subset G$ of cardinality $non(\mathcal{N})$ such that $S \times S$ is not naively Haar null in $G \times G$?

Finally, we present a ZFC example of the failure of the $\sigma$-idealness in $\R^2$.

\begin{example}
 Davies \cite{davies} proved that $\R^2$ can be covered by countably many rotated graphs of functions. But a graph of a function is naively Haar null: an arbitrary continuous Borel probability measure concentrated on the range axis of the function will be a witness measure. 
\end{example}
 
\begin{question}
Is it possible to use Davies' construction to obtain an example to the failure of the $\sigma$-idealness in every abelian Polish group?
\end{question}

%\newpage
\bigskip

M\'arton Elekes

Alfr\'ed R\'enyi Institute of Mathematics

Hungarian Academy of Sciences

P.O. Box 127, H-1364 Budapest, Hungary

elekes.marton@renyi.mta.hu

www.renyi.hu/ $\tilde{}$ emarci

and

E\"otv\"os Lor\'and University

Department of Analysis

P\'az\-m\'any P. s. 1/c, H-1117, Budapest, Hungary

\bigskip

Zolt\'an Vidny\'anszky

Alfr\'ed R\'enyi Institute of Mathematics

Hungarian Academy of Sciences

P.O. Box 127, H-1364 Budapest, Hungary

vidnyanszky.zoltan@renyi.mta.hu

www.renyi.hu/ $\tilde{}$ vidnyanz

\end{document}